\newif\iffascist\fascistfalse

\documentclass{amsart}
\usepackage{mathabx,amsfonts,amssymb,amsrefs,tikz,lineno}
\usetikzlibrary{decorations.pathmorphing}

\iffascist
\def\coordinate{coordinate}
\def\coefficient{coefficient}
\else
\def\coordinate{co\"ordinate}
\def\coefficient{co\"efficient}
\fi

\newcommand\Z{\mathbb{Z}}
\newcommand\N{\mathbb{N}}
\newcommand\R{\mathbb{R}}
\newcommand\Grig{{G_{012}}}

\newcommand\stab{{\operatorname{stab}}}
\newcommand\supp{{\operatorname{supp}}}
\newcommand\sym{{\mathfrak S}}
\newcommand\expect{{\mathbf E}}
\newcommand\proba{{\mathbf P}}
\newcommand\pair[1]{{\langle\!\langle}#1{\rangle\!\rangle}}

\newtheorem{lemma}{Lemma}[section]

\newtheorem{proposition}[lemma]{Proposition}
\newtheorem*{theorem*}{Theorem}
\newtheorem{theorem}[lemma]{Theorem}
\newtheorem{corollary}[lemma]{Corollary}

\newtheorem{maintheorem}{Theorem}

\theoremstyle{definition}
\newtheorem{example}[lemma]{Example}
\newtheorem{remark}[lemma]{Remark}
\newtheorem{definition}[lemma]{Definition}
\newtheorem{question}[lemma]{Question}

\font\manfnt=manfnt\newcommand\bend[1]{{$\odot$\marginpar{\manfnt\char127}}{#1}$\odot$}

\begin{document}
\title{Poisson-Furstenberg boundary and growth of groups}

\author{Laurent Bartholdi}
\email{laurent.bartholdi@gmail.com}
\address{\'Ecole Normale Sup\'erieure, Paris \emph{and} Georg-August Universit\"at zu G\"ottingen}

\author{Anna Erschler}
\email{anna.erschler@ens.fr}
\address{\'Ecole Normale Sup\'erieure, Paris}

\date{12 January 2015} 

\thanks{The work is supported by the ERC starting grant GA 257110
  ``RaWG'', the Courant Research Centre ``Higher Order Structures''
  of the University of G\"ottingen, and the ANR grant ANR-14-ACHN-0018-01}

\begin{abstract}
  We study the Poisson-Furstenberg boundary of random walks on
  permutational wreath products. We give a sufficient condition for a
  group to admit a symmetric measure of finite first moment with
  non-trivial boundary, and show that this criterion is useful to
  establish exponential word growth of groups. We construct groups of
  exponential growth such that all finitely supported (not necessarily
  symmetric, possibly degenerate) random walks on these groups have
  trivial boundary. This gives a negative answer to a question of
  Kaimanovich and Vershik.
\end{abstract}
\maketitle

\section{Introduction}
Consider a set $X$ with a basepoint $\rho$ and a right action of a
group $G$. A \emph{random walk} on $X$ is defined by a probability
measure $\mu$ on $G$; the random walker starts at $\rho$ and, at each
step, moves from $x$ to $x g$ with probability $\mu(g)$.  An important
particular case is $X=G$, seen as a $G$-space under
right-multiplication. A measure $\mu$ is \emph{symmetric} if
$\mu(g)=\mu(g^{-1})$ for all $g\in G$, and is \emph{non-degenerate} if
its support generates the group $G$.

The Poisson-Furstenberg boundary is the space of ergodic components of
the time shift in the space of infinite trajectories of the walk.  In
the case of random walks on groups, there are several equivalent
definitions of the Poisson-Furstenberg boundary, and we recall some of
them in Section~\ref{sec:def}. For more information, see e.g.\
\cites{kaimanovich-v:entropy}.

There is a strong relation between triviality/non-triviality of the
boundary and other asymptotic properties of groups
(see~\cites{kaimanovich-v:entropy}, and~\cites{erschler:hyderabad} for
a more recent overview).  For example, a result of Kaimanovich-Vershik
and Rosenblatt~\cites{kaimanovich-v:entropy,rosenblatt:mixing} states
that a group is amenable if and only if it admits a non-degenerate
measure with trivial boundary.

A natural question to ask is: ``Can exponential word growth be
characterized by non-triviality of the boundary for appropriate
measures?''  Indeed, there are several manifestations of the analogy
between non-triviality of the boundary and exponential growth, such as
the ``Entropy Criterion'' of Derriennic~\cite{derriennic:thergodique}
and Kaimanovich-Vershik~\cite{kaimanovich-v:entropy} and the
``Choquet-Deny'' theorem for nilpotent groups, proven by Dynkin and
Maliutov~\cite{dynkin-maljutov:choquetdeny}. There are partial results
towards the characterisation questions explained below, such as
\cite{erschler:liouville}*{Theorem 4.1} stating that a solvable group
$G$ admits a symmetric measure with non-trivial Poisson boundary if
and only if $G$ has exponential growth.

The ``Entropy Criterion'' implies in particular that if a group admits a
finitely supported measure with non-trivial boundary, then its word
growth is exponential.

Kaimanovich and Vershik conjecture
in~\cite{kaimanovich-v:entropy}*{page~466} that every group of
exponential growth admits a symmetric (possibly infinitely supported)
measure with non-trivial boundary, and add: ``It is plausible that
such a measure can be chosen finitary (but non-symmetric)''.

Their conjecture remains open; but we show in
Section~\ref{sec:trivboundary} that their addendum does not hold:
\begin{maintheorem}\label{thm:mainA}
  There exists a finitely generated group $G$ of exponential word
  growth, such that the boundary of $(G,\mu)$ is trivial for all
  finitely supported (possibly degenerate and non-symmetric) measures
  $\mu$.
\end{maintheorem}

There are many examples of groups of exponential growth such that
every symmetric finitely supported measure has trivial boundary, yet
some finitely supported measure has non-trivial boundary; e.g., wreath
products of a non-trivial finite group with $\Z$ or $\Z^2$,
see~\cite{kaimanovich-v:entropy}*{\S6.4}, and solvable
Baumslag-Solitar groups,
see~\cite{kaimanovich-v:entropy}*{\S6.6}. Likewise, on some groups
(e.g.\ wreath products of a finite group with the infinite dihedral
group, see Gilch~\cite{gilch:acceleration}), every non-degenerate
measure has trivial boundary, but these groups admit degenerate
measures with non-trivial boundary.

The groups we construct to prove Theorem~\ref{thm:mainA} are
\emph{permutational wreath products}, as defined in the next section:
groups $W=A\wr_XG=\sum_XA\rtimes G$. More precisely, we consider
a family of groups $W=A\wr_{X_1\times\cdots\times
  X_d}(G_1\times\cdots\times G_d)$, in which each $G_i$ is a copy of
the first Grigorchuk group, each $X_i$ is an orbital Schreier
graph of the first Grigorchuk group, and $A$ is a finite group. See
Sections~\ref{sec:def} and~\ref{sec:ntboundary} for details of this
construction. See also Section~\ref{sec:grig} for a larger family of
examples.

Ordinary wreath products ($X=G$) have exponential growth as soon as
$A$ is non-trivial and $G$ is infinite, but the situation is more
subtle for permutational wreath products, which may have intermediate
growth.  Indeed, it is shown
in~\cite{bartholdi-erschler:permutational} that the $W$ above has
intermediate growth if $d=1$. In this paper we consider the groups $W$
with $d\ge2$, and are in particular interested in the case $d=2$.  For
$d\ge 3$ one can show that a simple random walk on such groups has
non-trivial Poisson-Furstenberg boundary, see
Example~\ref{ex:d=3}. For the proof of Theorem~\ref{thm:mainA} we
consider the case $d=2$, which lies in some sense on a borderline
between exponential and intermediate growth: on one hand, as we
explain below, the growth is exponential. On the other hand, these
groups are ``close to groups of subexponential growth'', in the sense
that any finitely supported measure on them has trivial boundary.

In many known examples of groups, their exponential growth can be
checked either by exhibiting a free semigroup (they exist in solvable
groups of exponential growth, and more generally in elementarily
amenable groups of exponential growth, see
Chou~\cite{chou:elementary}); or by proving the existence of an
imbedded regular tree in the group's Cayley graph (as is the case for
any non-amenable group~\cite{benjamini-schramm:trees}). Ordinary
(non-permutational) wreath products of a non-trivial group by an
infinite group also contain regular trees in their Cayley graph, and
this class of groups contains interesting examples of torsion groups
of exponential growth, see Grigorchuk~\cite{grigorchuk:superamenable}.

Our understanding is that, for the groups we consider in this paper,
it is not straightforward to check that their growth is exponential.
To show that our examples have exponential growth we prove (in
Section~\ref{sec:ntboundary}) the following criterion based on random
walks:

\begin{maintheorem}\label{thm:mainB}
  Let $G$ be a finitely generated group acting on a set $X$, and let
  $\mu$ be a finitely supported, symmetric, non-degenerate measure on
  $G$. Suppose that the drift function of $\mu$ is at most
  $D n^\alpha$ for some constants $\alpha<1$ and $D$, and that, for
  every $\rho\in X$, the probability of return to $\rho$ of the
  $n$-step random walk induced on $X$ is at most $C n^{-\delta}$, for
  some constants $\delta>\alpha$ and $C$. Let $A$ be a non-trivial
  group.

  Then $W:=A\wr_X G$ admits a symmetric measure with finite first
  moment and non-trivial Poisson-Furstenberg boundary.  In particular,
  the word growth of $W$ is exponential.
\end{maintheorem}

In our situation, we consider $\delta=1$ and some $\alpha<1$. The
assumption of Theorem~\ref{thm:mainB} on the probability to return to
the origin on $X$ is a consequence of the fact that $X$ is a product
of two copies of infinite transitive Schreier graphs. For our main
examples used for the proof of Theorem~\ref{thm:mainA}, the condition
on the drift in Theorem~\ref{thm:mainB} follows from an upper bound on
the growth of Grigorchuk groups. To get more examples of this kind, we
consider groups for which the condition on the drift on $G$ is ensured
by a version of a self-similar-random-walk argument due to
Bartholdi-Virag and Kaimanovich; see
Example~\ref{ex:grigorchukgeneral}. We also give a torsion-free group
with this property, see Example~\ref{ex:grigorchuktf}.

On the other hand, to prove that the random walks we consider have
trivial Poisson-Furstenberg boundary, we use a criterion due to
Kalpazidou and Mathieu ensuring recurrence of ``centered'' random
walks, and a criterion for triviality of the boundary of random walks
on permutational wreath products
(Proposition~\ref{prop:recurrenceimpliestriv}). This criterion is
more complicated than in the case of ordinary wreath
products, see the discussion at the beginning of
Section~\ref{sec:trivboundary}.

The groups we construct in this paper admit a (symmetric,
non-degenerate) finite first moment measure with non-trivial boundary.
This leads us to ask the following question:
\begin{question}
  Does there exist a group $G$ of exponential word growth, such that
  all (not necessarily symmetric, not necessarily non-degenerate)
  measures with finite first moment have trivial Poisson-Furstenberg
  boundary?
\end{question}

\section{Definitions and preliminaries}\label{sec:def}
\subsection{Poisson-Furstenberg boundary, entropy and drift}
Consider two infinite trajectories $\boldsymbol x$ and $\boldsymbol
y$. We say they are \emph{equivalent} if they coincide after some
instant, possibly up to the time shift: there exists $N\in\N,k\in\Z$
such that $\boldsymbol x_n=\boldsymbol y_{n+k}$ for all $n\ge
N$. Consider the measurable hull of this equivalence relation in the
space of infinite trajectories. The quotient by this equivalence
relation is called the \emph{Poisson-Furstenberg boundary}.

A function $F\colon G \to \R$ is called $\mu$-\emph{harmonic} if for all
$g\in G$ we have $F(g) = \sum_{h\in G} F(g h) \mu(h)$.  The
Poisson-Furstenberg boundary is non-trivial if and only if $G$ admits
a bounded $\mu$-harmonic function which is non-constant on the group
generated by the support of $\mu$.

The \emph{entropy} of a probability measure $\mu$ is computed as
$H(\mu)=-\sum_g \mu(g) \log(\mu(g))$.  The \emph{entropy of the random
  walk}, also called its \emph{asymptotic entropy}, is the limit
$h(\mu)$ of $H(\mu^{*n})/n$, as $n$ tends to infinity. This limit is
well-defined, since the function $H(n):=H(\mu^{*n})$ is subadditive.
If $H(\mu)=\infty$, then $H(\mu^{*n})=\infty$ for all $n$ and in this
case we put $h(\mu)=\infty$.

Fix a finite generating set $S$ and consider on $G$ the word metric
$\|\cdot\|_S$ associated with $S$. Given a probability measure $\mu$
on $G$ and $\beta >0$, the \emph{$\beta$-moment} of $\mu$ with respect
to $S$ is $L^\beta(\mu):=\sum_{g\in G} \mu(g) \|g\|_S^\beta$. Clearly,
if the $\beta$-moment is finite with respect to some finite generating
set $S$, then it is finite with respect to any other generating
set. The \emph{first moment} of $\mu$ is simply written $L(\mu)$.

 The function $L(n):=L(\mu^{*n})$ is also
subadditive, by the triangular inequality for $\|\cdot\|_S$. It
expresses the mean distance to the origin in the word metric
$\|\cdot\|_S$, after $n$ steps of the random walk.
The \emph{drift}, also called \emph{rate of escape}, of the random
walk $(G,\mu)$ is the limit $\ell(\mu)$ of $L(n)/n$ as $n$ tends to
infinity; this limit is well-defined because $L(n)$ is subadditive. If
the first moment of $\mu$ is finite, that is, if $L(1)<\infty$, then
$\ell(\mu) \le L(1) <\infty$.

The \emph{entropy criterion} (Derriennic,
Kaimanovich-Vershik~\cites{derriennic:thergodique,kaimanovich-v:entropy})
states that if $\mu$ is a measure of finite entropy, then the boundary
of $(G,\mu)$ is trivial if and only if the entropy of the random walk
$h(\mu)$ is zero. If $\ell(\mu)=0$ then $h(\mu)=0$. For symmetric
measures, the converse is true: a symmetric measure $\mu$ of finite
first moment has zero drift ($\ell(\mu)=0$) if and only if the entropy of
the random walk $(G,\mu)$ is zero ($h(\mu)=0$), see
Karlsson-Ledrappier~\cite{karlsson-ledrappier:drift}.

We say that a measure $\mu$ is \emph{non-degenerate} if its support
generates $G$.  If the boundary of $\mu$ is trivial, then the group
generated by the support of $\mu$ is amenable. Every amenable group
admits a non-degenerate measure with trivial boundary
(Kaimanovich-Vershik,
Rosenblatt~\cites{kaimanovich-v:entropy,rosenblatt:mixing}); this
measure can be chosen symmetric and with support equal to $G$.

\subsection{Random walks on permutational wreath products}
We consider groups $A$, $G$, such that $G$ acts on a set $X$ from the
right. The \emph{(permutational) wreath product} $W=A\wr_X G$ is the
semidirect product of $\sum_X A$ with $G$. The \emph{support}
$\supp(f)$ of a function $f\colon X\to A$ consists in those $x\in X$ such
that $f(x)\neq1$.  The restricted product $\sum_X A$ is the group of finitely
supported functions $X\to A$.  The left action of $G$ on $\sum_X A$ is
then defined by $(g\cdot f)(x) = f(x g)$; observe that for $g_1$,
$g_2$ in $G$
\[
(g_1g_2\cdot f)(y) = f(y(g_1g_2)) = f((y g_1)g_2) = (g_2\cdot f)(y g_1) = (g_1\cdot g_2\cdot f)(y).
\]
We have in particular $\supp(g^{-1}\cdot f)=\supp(f)g$.

If $A$ and $G$ are finitely generated and if the action of $G$ on $X$
is transitive, then the permutational wreath product is a finitely
generated group. Indeed, fix finite generating sets $S_A$ and $S_G$ of
$A$ and $G$ respectively, and fix a basepoint $\rho\in X$. The wreath
product is generated by $S=S_A \cup S_G$.  Here and in the sequel we
identify $G$ with its image in $W$ under the imbedding $g \to (1,g)$ and
identify $A$ with its image in $W$ under the imbedding $a \to (f_a,1)$, where
$f_a\colon X \to A$ is defined by $f_a(\rho)=a$ and $f(x)=1$ for all $x\ne
\rho$. We call $S$ the \emph{standard generating set} of $W$ defined
by $S_A$, $S_G$.  Analogously, if the action of $G$ on $X$ has
finitely many orbits, then $W$ is finitely generated by $S_G\cup
(S_A\times X/G)$. If the action has infinitely many orbits, then the
permutational wreath product is not finitely generated.

The Cayley graph of the permutational wreath product with respect to
the standard generating set $S$ can be described as follows. Elements
of $W=\sum_X A\rtimes G$ are written $f g$ with $f\in\sum_X A$ and $g\in
G$; multiplication is given by $(f_1g_1)(f_2g_2)=f_1(g_1\cdot
f_2)g_1g_2$.

Consider a word $v=s_1s_2 \dots s_\ell$, with all $s_i \in S$, and
write its value in $W$ as $f_v g_v$. Set $u=f_u g_u= s_1 s_2 \dots
s_{\ell-1}$. Here $g_u, g_v$ belong to $G$, and $f_u, f_v$ belong to
$\sum_X A$.

We consider two cases, depending on whether $s_\ell\in S_A$ or
$S_\ell\in S_G$. If $s_\ell\in S_A$, we have an edge of ``A'' type
from $u$ to $v$. The multiplication formula gives $g_v=g_u$ and
$f_v(x)=f_u(x)$ for all $x\neq \rho g_u^{-1}$, while
$f_v(\rho g_u^{-1})=f_u(\rho g_u^{-1})s_\ell$.

If $s_\ell\in S_G$, we have an edge of ``G'' type from $u$ to
$v$. In that case, $f_v=f_u$, and $g_v=g_us_\ell$.

We have begun to study asymptotic properties of permutational wreath
products in~\cite{bartholdi-erschler:permutational}.  The asymptotic
geometry of these groups turns out to be much richer than in the
particular case of ordinary wreath products (namely, for which
$X=G$). It is easy to see that the word growth of $A \wr G$ is
exponential whenever $X=G$ is infinite and $A$ is non-trivial.
However, among permutational wreath products there are groups of
intermediate growth, see~\cite{bartholdi-erschler:permutational}.

Given a probability measure $\mu$ on $W=A\wr_X G$, we say that the
random walk is \emph{translate-or-switch} if the support of $\mu$
belongs to the union of $G$ and $A$; in other words, $\mu=p\mu_A +q
\mu_G$, where $p, q\ge 0$, $p+q=1$, the support of $\mu_A$ belongs to
$A$, and the support of $\mu_G$ belongs to $G$.

We say that the random walk is \emph{switch-translate-switch} if
$\mu=\mu_A * \mu_G *\mu_A$, for measures $\mu_A,\mu_G$ supported on
$A$ and $G$ respectively.  If $X=G$, the ``switch-translate-switch''
random walks are called ``switch-walk-switch''. In this case, we can
view each step of the random walk as follows: we ``switch'' the value
of the configuration at the point where the random walker stands, then
make one step of the random walk on $G$ and then switch the
configuration at the point of the arrival. Note however that no such
interpretation is valid for a general permutational wreath product,
because translation and movement are in general genuinely different
operations.

Let $w=g_1\dots g_n$ be a word over $G$ of length $n=|w|$, and let
$\rho\in X$ be a base point. The \emph{inverted orbit} of $w$ is the
set $\{\rho,\rho g_n,\rho g_{n-1}g_n,\dots,\rho g_1\cdots g_n\}$; the
\emph{inverted orbit growth} is the cardinality
\[\delta_\rho(w)=\#\{\rho,\rho g_n,\rho g_{n-1}g_n,\dots,\rho g_1\cdots g_n\}\]
of that set. In the sequel, $\rho$ will be fixed, and we will usually omit
it from the notation.

If $w$ is a word corresponding to a length-$n$ trajectory of a random walk, then we can consider
$\delta_\rho(w)$ as a random variable with values in $\{1,\dots,n+1\}$.

\section{Criteria for non-triviality of the  boundary}
We characterize in this section groups with non-trivial
Poisson-Furstenberg boundary, with the goal of applying it to
permutational wreath products. For ordinary wreath products, a
well-known criterion by Kaimanovich and
Vershik~\cite{kaimanovich-v:entropy}*{Proposition~6.4} states that,
for $A\neq1$ finite and finitely supported measures, $A\wr G$ has
trivial boundary if and only if the projection of the random walk to
$G$ is recurrent. We extend this criterion to permutational wreath
products.

\begin{lemma}\label{lem:expectedinvertedorbit}
  Let the group $G$ act on $X$ and let $\mu$ be a probability measure
  on $G$. Let $\rho\in X$ be a basepoint. Then the induced random walk
  on $X$ starting at $\rho$ is recurrent if and only if the expectancy of the inverted
  orbit growth $\expect[\delta_\rho(w)]$ is sublinear in $|w|$.
\end{lemma}
\begin{proof}
  Consider the random variable $A_{i,n}$ which equals $1$ if the
  $i$-th point on the inverted orbit of the trajectory of the random
  walk is distinct from any points from $1$ to $i-1$, and equals $0$
  otherwise. Consider a word $w=g_1\dots g_n$.  Since $G$ acts by
  permutations on $X$, we have
  \begin{align*}
    w\in A_{i,n} &\Leftrightarrow \rho g_i \cdots g_n \not\in\{\rho g_{i-1} \cdots g_n, \rho g_{i-2} \cdots g_n,\dots,\rho g_1 \cdots g_n\}\\
    &\Leftrightarrow \rho \not\in\{\rho g_{i-1}, \rho g_{i-2} g_{i-1},
    \dots, \rho g_1 \cdots g_{i-1}\}\\
    &\Leftrightarrow \rho \not\in\{\rho g_{i-1}^{-1},\rho g_{i-1}^{-1} g_{i-2}^{-1}, \dots,\rho g_{i-1}^{-1} g_{i-2}^{-1} \cdots g_1^{-1}\}.
  \end{align*}
  Therefore,
  \[\expect[A_{i,n}] = \proba[\rho\neq\rho g_{i-1}^{-1}\text{ and }\rho\ne\rho
  g_{i-1}^{-1} g_{i-2}^{-1}\text{ and }\dots\text{ and }\rho \ne \rho
  g_{i-1}^{-1} g_{i-2}^{-1} \cdots g_1^{-1}];
  \]
  observe that this is the probability $p_i$ that the random walk on
  $X$ induced by $(G,\check\mu)$ with
  $\check\mu(g)=\mu(g^{-1})$, starting from $\rho$, never returns to
  this base point $\rho$ during time moments between $1$ and
  $i-1$. Note that for each $i$ we have $p_i\ge p_{i+1}$.  If the
  random walk on $X$ induced by $(G,\check\mu)$ is recurrent, then
  $p_i \to 0$, while if the random walk on $X$ induced
  by$(G,\check\mu)$ is transient, then there exists $p>0$ such
  that $p_i\ge p$ for all $i$.

  Next, observe that the expectation of $\delta_\rho(w)$ is
  \begin{align}\label{eq:expdelta}
    \expect[\delta_\rho(w)]&=\expect[1+A_{1,n}+A_{2,n}+ \dots +A_{n,n}]\\
    &= 1+\expect[A_{1,n}]+\expect[A_{2,n}]+ \dots +\expect[A_{n,n}].\notag
  \end{align}
  Therefore, the expectation of $\delta_\rho(w)$ grows linearly (at least $p n$)
  if the random walk is transient, while $\delta_\rho(w)/|w|$ tends to zero if
  the random walk is recurrent.

  Finally, observe that $(G,\check\mu)$ induces a transient random walk
  on $X$ if and only if the stabilizer $G_\rho$ of $\rho \in X$ is a
  transient set for the random walk $(G,\check\mu)$; because
  $G_\rho=G_\rho^{-1}$, this happens if and only if $(G,\mu)$ induces
  a transient random walk on $X$.
\end{proof}

We continue now with two propositions, each giving a sufficient
condition for non-triviality of the boundary. One uses finiteness of
the entropy of $\mu$, and the other a weak restriction on the support
of $\mu$. The proof of Lemma~\ref{lem:lineartimpliesnottriv}
appears further below.
\begin{lemma}\label{lem:lineartimpliesnottriv}
  Let $\mu$ be a non-degenerate measure on $W$ with finite entropy
  $H(\mu)$, and suppose that the expected size of the inverted orbits
  of the random walk defined by the projection of $\mu$ to $G$ grows
  linearly.  Then the entropy $h(\mu)$ of the random walk $(W,\mu)$
  is positive; so the Poisson-Furstenberg boundary of $(W,\mu)$ is
  non-trivial.
\end{lemma}

\noindent Combining Lemma~\ref{lem:expectedinvertedorbit} and
Lemma~\ref{lem:lineartimpliesnottriv}, we get the following
\begin{proposition}\label{prop:transientimpliesnottriv}
  Let $\mu$ be a non-degenerate random walk on $W$ with finite entropy
  and transient projection to $X$. Then the Poisson-Furstenberg
  boundary of the random walk $(W,\mu)$ is non-trivial.
\end{proposition}

\begin{proof}[Proof of Lemma~\ref{lem:lineartimpliesnottriv}]
  The argument is similar to  that of
  \cite{erschler:liouville}*{Theorem 3.1}.

  Let $\rho\in X$ be any basepoint. First, there exists $N\in\N$ and
  $f\neq f'\in\sum_X A$ both supported on $\{\rho\}$ such that
  $f,f'\in\supp(\mu^{*N})$, because $\mu$ is non-degenerate.  Recall
  that we identify $A$ with those functions $f\in\sum_X A$ that are
  supported on $\{\rho\}$.  Since $h(\mu^{*N})=N h(\mu)$, it suffices
  to prove $h(\mu^{*N})>0$; up to replacing $\mu$ by $\mu^{*N}$, we
  suppose, from now on, that there are at least two elements in
  $A\cap\supp(\mu)$; in particular, $\mu(A)>0$. If at time instant $n$
  the increment of the random walk belongs to $A$, we say that at this
  instant there is an `$A$' \emph{step} of the random walk.  Define
  the normalized measure $\nu\colon A\to\R$ by $\nu(a)=\mu(a)/\mu(A)$;
  by assumption, $H(\nu)>0$.

  Let $w=w_1\dots w_n\in W^n$ be a trajectory, and let $W$ act on $X$ via
  the quotient map $W\to G$. Consider, in analogy with the proof of
  Lemma~\ref{lem:expectedinvertedorbit}, the event $Z_{i,n}$ that the
  inverted random walk of $G$ on $X$ visits a new point at time $i$, namely
  \[Z_{i,n}\Leftrightarrow\rho w_i\cdots w_n\not\in\{\rho w_j\cdots w_n\mid j>i\}.\]
  As in~\eqref{eq:expdelta} we have
  \[\expect[\delta_\rho(w)]=\proba[Z_{1,n}]+\cdots+\proba[Z_{n,n}]+1.\]
  For $i\ge2$, let $Z'_{i,n}=Z_{i,n}\wedge(w_{i-1}\in A)$ denote the event
  that $Z_{i,n}$ holds and that an `$A$' step is performed at the moment
  the new point is visited. Since $Z_{i,n}$ and $(w_{i-1}\in A)$ are
  independent events, $\proba[Z'_{i,n}]=\mu(A)\proba[Z_{i,n}]$. Therefore,
  the expected number $t$ of distinct points $x_1,\dots,x_t\in X$ that
  belong to the inverted orbit of $w$ and at which an `$A$' step occurs is
  at least
  \[\mu(A)\proba[Z_{2,n}]+\cdots+\mu(A)\proba[Z_{n,n}]+\mu(A)\ge\mu(A)\expect[\delta_\rho(w)]-1.\]

  Since the expectation of $\delta_\rho(w)$ is linear in $n$, the
  expectation of $t$ is also linear in $n$, say $\expect[t]\ge d n$
  for all $n$ large enough and a constant $d>0$.

  Now given a random trajectory $w=w_1w_2\dots w_n\in W^n$, with steps
  $w_i$, let $i_1,\dots,i_t\in\{1,\dots,n-1\}$ denote those times $i$
  at which $Z'_{i+1,n}$ holds. Let $\tau$ denote the partition of
  $W^n$ in which $w=w_1\dots w_n$ and $w'=w'_1\dots w'_n$ belong to
  the same part if they have the same $t$, same indices
  $i_1,\dots,i_t$ and if $w_i=w'_i$ for all
  $i\not\in\{i_1,\dots,i_t\}$. In particular all
  $w_{i_1},\dots,w_{i_t}$ belong to $A$, and if for $s=1,\dots,t$ we
  let $x_s=\rho w_{i_{s}}\cdots w_n$ denote the corresponding points
  in the inverted orbit, then they are the same for all elements of a
  part.

  We compute the conditional entropy $H(\mu^{*n}\mid\tau)$. All
  elements $w$ in a part of the partition are obtained by selecting
  independently $t$ elements $w_{i_1},\dots,w_{i_t}\in A$ according to
  the normalized measure $\nu$. All these words $w$ evaluate to
  distinct elements $\overline w$ of $W$: each $w_{i_1},\dots,w_{i_t}$
  may be recovered from $\overline w\in W$ by writing
  $\overline w=(f,g)$ with $f\colon X\to A,g\in G$ and considering
  $f(x_1),\dots,f(x_t)$. Therefore,
  \[H(\mu^{*n}\mid \tau)\ge\expect[t]H(\nu)\ge \mu(A)d n H(\nu);
  \]
  we have $H(\mu^{*n})\ge H(\mu^{*n}\mid\tau)$, because entropy is
  not less than the mean conditional entropy, see
  e.g.~\cite{rohlin:lectures}*{\S5}, so
  \[h(\mu)=\lim\frac1n H(\mu^{*n})\ge\lim\frac1n
  H(\mu^{*n}\mid\tau)\ge \mu(A) d H(\nu)>0.\qedhere\]
\end{proof}

We consider now a different sufficient condition for non-triviality of
the boundary, requiring only a very weak form of non-degeneracy of the
random walk:
\begin{proposition}\label{prop:additional}
  Let $\mu$ be ``switch-translate-switch'' random walk on $W$, and
  assume that there exist $n\in\N$ and two elements in the support of
  $\mu^{*n}$ with equal projection to $G$. Assume also that the
  projected random walk $(X,\overline\mu)$ is transient.

  Then the Poisson-Furstenberg boundary of the random walk $(W,\mu)$
  is non-trivial.
\end{proposition}
Note, in particular, that the first condition holds as soon as $\mu$
is non-degenerate and $A\neq1$.  Note also that the `translation' part
of $\mu$ is allowed to be infinitely supported on $X$. Indeed, we will
later apply Proposition~\ref{prop:additional} to infinitely supported
measures.

\begin{proof}
  Consider a trajectory $\Theta=(1,f_1g_1,f_2g_2,\dots)$ of the random
  walk on $W$, with $f_i\in\sum_X A$ and $g_i\in G$, and
  $(f_i g_i)^{-1}f_{i+1}g_{i+1}\sim\mu$.

  By assumption, there exists $n\in\N$ and $u,v\in\sum_XA$, $g\in G$
  such that two elements $ug\neq vg\in W$ are reached with positive
  probability at time $n$ of the walk. Choose a \coordinate\
  $\sigma\in X$ in which $u$ and $v$ differ, say $v(\sigma)=a
  u(\sigma)$ for some $a\neq1$ in $A$.

  Consider $f_i(\sigma)$. Since the random walk we consider is
  ``switch-translate-switch'', for all $i$ the elements $f_i$ and
  $f_{i+1}$ differ in at most two places. More precisely,
  $f_i(\sigma)\neq f_{i+1}(\sigma)$ only when $\sigma\in\{\rho
  g_i^{-1},\rho g_{i+1}^{-1}\}$.

  Since $\overline\mu$ is transient, there is a bound $R\in\N$ such that,
  almost surely, we have $\sigma\in\{\rho g_i^{-1},\rho
  g_{i+1}^{-1}\}$ in at most $R$ instants $i$. It follows that
  \[\phi_\sigma(\Theta):=\lim_{i\to\infty}f_i(\sigma)\]
  almost surely exists, and defines a measurable function on the space
  of trajectories.

  For any $\epsilon>0$, at least $1-\epsilon$ of the mass of $\mu$ is
  concentrated on a finite set $W_\epsilon\subset W$; there exists a
  finite subset $A_\epsilon\subseteq A$ such that $f\in\sum_X
  A_\epsilon$ whenever $f g\in W_\epsilon$; so, with probability
  $1-R\epsilon$, the limit $\phi_\sigma$ belongs to the finite set
  $A_\epsilon^R$. Take now $\epsilon$ small enough so that
  $R\epsilon<\mu^{*n}(ug)$. Then there exists $b\in A_\epsilon^R$ such
  that with positive probability the trajectory $\Theta$ visits $ug$
  at time $n$ and satisfies $\phi_\sigma(\Theta)=b$.

  For each such trajectory, replace the initial $n$ steps (reaching
  $ug$) with an $n$-step random walk reaching $vg$. This produces a
  positive-measure set of trajectories that visit $vg$ at time $n$ and
  satisfy $\phi_\sigma(\Theta)=v(\sigma)u(\sigma)^{-1}b=ab\neq
  b$. Therefore, $\phi_\sigma$ is not constant.

  We have exhibited a non-constant measurable function on the space of
  exits of the random walk, so its boundary is not trivial.
\end{proof}

Alternatively, as a replacement for the last three paragraphs of the
proof, let the element $a=v(\sigma)u(\sigma)^{-1}\in A$ have order
$m\in\N\cup\{\infty\}$. Let $T$ be a right transversal of
$\langle a\rangle$ in $A$; that is, $A=T\sqcup aT\sqcup\cdots$. If
$m=\infty$, set $A_0=\bigsqcup_{n\in\Z}a^{2n}T$ and
$A_1=\bigsqcup_{n\in\Z}a^{2n+1}T$, while if $m<\infty$, set $A_n=a^nT$
for all $n\in\{0,\dots,m-1\}$. Then the function
$\chi:\Theta\mapsto(n\text{ if }\phi_\sigma(\Theta)\in A_n)$ is
measurable, takes finitely many values, and takes value $n$ with
positive probability if and only if it takes value $n-1\pmod m$ or
$\pmod 2$ with positive probability, so is not constant.

\section{Groups admitting finite first moment measures with non-trivial Poisson-Furstenberg boundary}\label{sec:ntboundary}
Our aim, in this section, is to prove that most wreath products of the
form $W=A\wr_{X_1\times X_2}(G_1\times G_2)$ have a non-trivial
boundary for an appropriate measure. This measure will be an infinite
convex combination of the convolutions powers of some symmetric
finitely supported measure on $W$.
 Our main task is to chose the
\coefficient s in the convex combinations in such a way that they decay
not to fast; on the other hand, they must decay fast enough that the
measure we construct has finite first moment.

To show that the measure has non-trivial boundary, we use the results
of the previous section
(Propositions~\ref{prop:transientimpliesnottriv}
and~\ref{prop:additional}).  At the end of this section, we give
applications of Theorem~\ref{thm:mainB}, and construct groups of
exponential growth that we will later use to prove
Theorem~\ref{thm:mainA}.

\begin{theorem}[= Theorem~\ref{thm:mainB}]\label{thm:expwordgrowth}
  Let $G$ be a group acting on a set $X$, and let $\mu$ be a finitely
  supported, symmetric, non-degenerate measure on $G$. Suppose that
  its drift satisfies $L_{\mu_i}(n)\le D n^\alpha$ for all $n\in\N$,
  for some constants $D$ and $\alpha<1$. Suppose also that, for every
  $\rho\in X$, the probability of return to $\rho$ satisfies
  $\mu^{*n}(\stab_G(\rho))\le C/n^\delta$ for all $n\in\N$, for some
  constants $C$ and $\delta>\alpha$. Let $A$ be a non-trivial group.

  Then $W:=A\wr_X G$ admits a symmetric measure with finite first
  moment and non-trivial Poisson-Furstenberg boundary.  In particular,
  the word growth of $W$ is exponential.
\end{theorem}

The idea of the proof is to construct a measure $\mu$, with finite
first moment, such that the induced random walk on $X$ is transient;
and then to use Proposition~\ref{prop:transientimpliesnottriv} or
Proposition~\ref{prop:additional} to conclude that the boundary of the
random walk $(W,\mu)$ is non-trivial.

\subsection{Reminder: properties of Stable Laws}

We start by recalling classical results on stable laws
from~\cite{ibragimov-linnik:iss}, which we restrict to measures on
$\R_+$.  A measure $\mu$ on $\R_+$ is \emph{stable} if for any
$a_1,a_2>0$ there are $a>0,b$ such that $\mu(a_1\cdot x)*\mu(a_2\cdot
x)=\mu(a\cdot x+b)$; in particular, if $\mu$ is the law of independent
random variables $X_1,X_2$, then the law of $X_1+X_2$ is an affine
transformation of $\mu$.

Let $X_1,X_2,\dots$ be independent random variables with law
$\mu'$. We say $\mu'$ is in the \emph{domain of attraction} of a
non-degenerate stable law $\mu$ if there are constants $A_n,B_n$ such
that the law $\mu_n$ of $(X_1+\cdots+X_n-A_n)/B_n$ converges weakly to
$\mu$; namely, if $\mu_n(M)\to\mu(M)$ for all Borel subsets
$M\subseteq\R$ whose boundary is $\mu$-negligible.

The \emph{distribution} of a measure $\mu$ on $\R_+$ is the function
$F(x)=\mu([0,x])$. Attraction towards a stable law can be checked by
estimating the regularity of the tails of the distribution:
\begin{theorem}[Part of \cite{ibragimov-linnik:iss}*{Theorem~2.6.1}]\label{thm:linnik1}
  A measure belongs to the domain of attraction of a stable law if and
  only if its distribution $F$ satisfies
  \[F(x)=1-\frac{h(x)}{x^\alpha}\text{ as }x\to\infty,\] for a
  function $h$ that varies slowly in the sense of Karamata (meaning $h(t
  x)/h(x)\to1$ for all $t>0$) and a parameter $\alpha\in(0,2)$. This
  parameter is called the \emph{exponent} of the measure.\qed
\end{theorem}

We will also use a local limit theorem due to Gnedenko: again, we only
quote a subcase of the general result.  Recall that a measure $\mu$ on
$\R$ has \emph{density} $g$ if $\mu(M)=\int_Mg(x)dx$ for all
measurable $M\subset\R$. All stable measures have a density, which
furthermore may be supposed to be a continuous (and therefore bounded)
function $g\colon\R\to\R$, see~\cite{ibragimov-linnik:iss}*{\S2.3}.


\begin{theorem}[Part of \cite{ibragimov-linnik:iss}*{Theorem~4.2.1}]\label{thm:linnik2}
  Suppose that $\mu'$ is supported on $\N$, but not on $h\N$ for any
  $h>1$, and suppose that $\mu'$ is in the domain of attraction of a
  stable law with density $g$. Then
  \[\lim_{n\to\infty}\sup_{k\in\N}\left|B_n(\mu')^{*n}(k)-g((k-A_n)/B_n)\right|=0.\qedhere\]
\end{theorem}

\subsection{Proof of Theorem~\ref{thm:expwordgrowth}}


\begin{proof}
  Consider a parameter $\gamma\in(1,2)$, to be fixed later. For
  $i\in\{1,2,\dots\}$, set $\alpha_i=C_\gamma/i^\gamma$ for a
  constant $C_\gamma$ defined such that $\sum_{i=1}^\infty \alpha_i=
  C_\gamma \sum_{i=1}^\infty i^{-\gamma} =1$. Define measures
  $\nu_\gamma$ on $\N$ and $\lambda_\gamma$ on $G$ by
  \[
  \nu_\gamma(i)=\alpha_i,\qquad\lambda_\gamma= \sum_{i=1}^\infty \alpha_i \mu^{*i}.
  \]
  By the definition of $C_\gamma$, both $\nu_\gamma$ and
  $\lambda_\gamma$ are probability measures. The following estimate on
  negative moments of $\nu_\gamma^{*n}$ will be needed later:
  \begin{lemma}\label{elementary}
    For all $\delta>0$ there exists a constant $C$ such that, for all
    $n\ge2$,
    \[\sum_{i=1}^\infty  \nu_\gamma^{*n}(i) /i^\delta \le\begin{cases}
      C/n^{1/(\gamma-1)} \log(n) & \text{ if }\delta=1,\\
      C/n^{\delta/(\gamma-1)} & \text{ if }\delta\neq1.
    \end{cases}
    \]
  \end{lemma}

  \begin{proof}
    We first show that $\nu_\gamma$ is in the domain of attraction of
    a stable law, by checking the hypotheses of
    Theorem~\ref{thm:linnik1}. Its distribution satisfies
    $1-F(x)=\sum_{x<i\in\N}C_\gamma i^{-\gamma}$, so
    \[\frac{C_\gamma}{\gamma-1}x^{1-\gamma}=\int_x^\infty C_\gamma
    t^{-\gamma}dt\le1-F(x)\le\int_{x+1}^\infty C_\gamma
    t^{-\gamma}dt=\frac{C_\gamma}{\gamma-1}(x+1)^{1-\gamma},
    \]
    so $C_\gamma/(\gamma-1)\le h(x)\le
    C_\gamma/(\gamma-1)(1+1/x)^{1-\gamma}$ and $h$ is slowly
    varying. Therefore, $\nu_\gamma$ is in the domain of attraction of
    a stable law of exponent $\alpha=\gamma-1$.

    It then follows from~\cite{ibragimov-linnik:iss}*{Theorem~2.1.1}
    that $B_n=n^{1/\alpha}h(n)$ for another function $h$ that slowly
    varies in the sense of Karamata.

    Let $g$ be the density of the stable law towards which
    $\nu_\gamma$ converges. By Theorem~\ref{thm:linnik2}, $\sup_k
    B_n(\nu_\gamma)^{*n}(k)-g((k-A_n)/B_n)$ converges to $0$ as
    $n\to\infty$, and $g$ is bounded, so $\sup_k
    B_n(\nu_\gamma)^{*n}(k)$ is bounded. Therefore, there exists a
    constant $C'$ such that $(\nu_\gamma)^{*n}(k)\le
    C'n^{-1/(\gamma-1)}$ for all $k\in\N$.

    We are now ready to prove the lemma. Set $a_n = n^{1/(\gamma-1)}$,
    and split the sum as
    \[\sum_{i=1}^\infty  \nu_\gamma^{*n}(i) /i^\delta  = \sum_{i=1}^{a_n}  \nu_\gamma^{*n}(i) /i^\delta +
    \sum_{i=a_n+1}^\infty \nu_\gamma^{*n}(i) /i^\delta.
    \]
    In the first summand, we use $\nu_\gamma^{*n}(i) \le C'/
    n^{1/(\gamma-1)}$ for all $i$, so
    \begin{align*}
      \sum_{i=1}^{a_n} \nu_\gamma^{*n}(i) /i^\delta &\le C'/
      n^{1/(\gamma-1)} \sum_{i=1}^{a_n} 1/i^\delta\\
      &\le C'/n^{1/(\gamma-1)}\begin{cases} \log{a_n} \le
        C''/n^{1/(\gamma-1)}\log(n) &\text{ if }\delta=1,\\
      a_n^{1-\delta}/(1-\delta) \le C''/n^{\delta/(\gamma-1)}&\text{ if
        }\delta\neq1,
      \end{cases}
    \end{align*}
    for some constant $C''$. For the second summand, we use the coarse
    estimate
    \[\sum_{i=a_n+1}^\infty \nu_\gamma^{*n}(i) /i^\delta \le 1/a_n^\delta
    \sum_{i=a_n+1}^\infty \nu_\gamma^{*n}(i) \le 1/a_n^\delta =
    1/n^{\delta/(\gamma-1)}
    \]
    and we are done, setting $C=C''+1$.
  \end{proof}

  Let us next find out for which $\gamma$ the random walk on $X$
  defined by $\lambda_\gamma$ is transient. The argument is close to
  that of \cite{erschler:critical}*{Lemma 3.1}.  In that lemma, it was
  shown that for any transitive action of $G$ on an infinite set $X$,
  the measures $\lambda_\gamma$ define a transient random walk on $X$
  as soon as $\gamma\in(1,3/2)$. For the proof of
  Theorem~\ref{thm:expwordgrowth}, however, it is not sufficient to
  work with $\gamma$ between $1$ and $3/2$, because the theorem's
  assumptions do not imply that $\lambda_\gamma$ has finite first
  moment for some $\gamma<3/2$. Indeed, we will use in an essential
  manner the additional assumption on the action to weaken the
  condition on $\gamma$.

  \begin{proposition}\label{prop:transient}
    Suppose that the probability of return to the origin $\rho\in X$
    for the random walk on $X$ induced by the measure $\mu$ satisfies
    $\mu^{*n}(\stab_G(\rho)) \le C/n^\delta$ for some $\delta>0$.
    Then, for all $\gamma\in(1,1+\delta)$, the random walk
    $(\lambda_\gamma,X)$ is transient.
  \end{proposition}
  \begin{proof}
    For any $H\subset G$, we have
    \[\lambda_\gamma^{*n}(H)= \sum_{i\ge 0} \nu_\gamma^{*n}(i) \mu^{*i} (H).
    \]
    In particular, this holds with $H$ the stabilizer of $\rho\in X$:
    \[\lambda_\gamma^{*n}(\stab_G(\rho))=\sum_{i\ge 1} \nu_\gamma^{*n}(i)
    \mu^{*i} (\stab_G(\rho)).
    \]
    By Lemma~\ref{elementary} we know that for any $\delta$
    
    \[\sum_{i=1}^\infty  \nu_\gamma^{*n}(i) /i^\delta \le C/n^{\delta/(\gamma-1)} \log(n).\]
    \noindent Therefore, for all $n\ge2$ we have
    \begin{align*}
      \lambda_\gamma^{*n}(\stab_G(\rho)) &= \sum_{i=1}^\infty \nu_\gamma^{*n}(i) \mu^{*i} (\stab_G(\rho))\\
      &\le\sum_{i=1}^\infty\nu_\gamma^{*n}(i)C/i^\delta\le C/n^{\delta/(\gamma-1)} \log(n).
    \end{align*}
    Since $\gamma<1+\delta$, we have $\delta/(\gamma-1)>1$, so
    \[\sum_{n=0}^\infty \lambda_\gamma^{*n}(\stab_G(\rho)) < \infty.\]

    This means that $\stab_G(\rho)$ is a transient subgroup for the
    measure $\lambda_\gamma$, or, in other words, that the random walk
    on $X$ induced by the measure $\lambda_\gamma$ is transient.
  \end{proof}

  We can now finish the proof of Theorem~\ref{thm:expwordgrowth}. It is
  time to use the assumption on the first moment of $\mu$.

  \begin{lemma}\label{lem:estimatemoment}
    For $\gamma>1+\alpha$, the first moment of the measure
    $\lambda_\gamma$ is finite.
  \end{lemma}

  \begin{proof}
    Recall that $L_\mu(i)$ denotes the first moment of the measure
    $\mu^{*i}$. By our assumption, there exists a constant $D$ such
    that $L_\mu(i)<Di^\alpha$ for all $i\in\N$.  The first moment of
    $\lambda_\gamma$ is therefore equal to
    \[\sum_{i=1}^\infty \nu_\gamma(i)L_\mu(i) \le D
    \sum_{i=1}^\infty C_\gamma i^{-\gamma} i^\alpha = D C_\gamma
    \sum_{i=1}^\infty i^{-(\gamma-\alpha)} < \infty,
    \]
    if $\gamma-\alpha>1$.
  \end{proof}

  Now fix $\gamma\in(1+\alpha,1+\delta)$. By
  Proposition~\ref{prop:transient}, the random walk $\lambda_\gamma$ is
  transient, while by Lemma~\ref{lem:estimatemoment} the first moment of
  $\lambda_\gamma$ is finite.

  Take a measure $\mu_A$ on $A$ with finite first moment, whose
  support contains $1$ and generates $A$. Set
  $\lambda=\mu_A*\lambda_\gamma*\mu_A$. Observe that $\lambda$ is a
  non-degenerate random walk with finite first moment, and that the
  induced random walk on $X$ is transient.  Therefore, by
  Proposition~\ref{prop:transientimpliesnottriv}, the boundary of
  $(W,\lambda)$ is non-trivial. This completes the proof of
  Theorem~\ref{thm:expwordgrowth}.

  Alternatively, note that $\lambda$ is a ``switch-translate-switch''
  random walk, so that Proposition~\ref{prop:additional} applies.
\end{proof}

\subsection{Consequences of Theorem~\ref{thm:expwordgrowth}}

\begin{corollary}\label{cor:mainB}
  Let $\alpha<1$ be given; for each $i=1,2$, let $G_i$ act
  transitively on an infinite set $X_i$, and let $\mu_i$ be a finitely
  supported, symmetric, non-degenerate probability measure whose drift
  satisfies $L_{\mu_i}(n)\le D n^\alpha$ for a constant $D$ and all
  $n\in\N$.

  Set $G=G_1\times G_2$ and $X=X_1\times X_2$, on which $G$ acts
  \coordinate wise. Let $A$ be a non-trivial group.

  Then $W:=A\wr_X G$ admits a symmetric measure with finite first
  moment and non-trivial Poisson-Furstenberg boundary.  In particular,
  the word growth of $W$ is exponential.
\end{corollary}
\begin{proof}
  Set $\mu=\mu_1\times\mu_2$; it is the random walk on $X$ that walks
  independently on $X_1$ and $X_2$. Choose a basepoint $\rho=(\rho_1,\rho_2)
  \in X$. Observe $\stab_G(\rho)\cap G_1 = \stab_{G_1}(\rho_1)$ and
  $\stab_G(\rho) \cap G_2 = \stab_{G_2}(\rho_2)$.  For all $n\ge 0$ we have
  $\mu^{*n}(\stab_G(\rho)) =\mu_1^{*n}(\stab_{G_1}(\rho_1))\mu_2^{*n}(\stab_{G_2}(\rho_2))$.

  We say that a symmetric random walk on a connected locally finite
  graph is a \emph{nearest neighbour} random walk if it is a
  symmetric random walk which walks along the edges of the graph with
  probability bounded away from zero: $p_1(x,y)=p_1(y,x)$,
  $p_1(x,y)>0$ implies that $x$ and $y$ are joined by an edge, and
  there exists $p>0$ such that $p_1(x,y) \ge p$ whenever $x$ and $y$
  are joined by an edge.

  For a nearest-neighbour symmetric random walk on a connected
  infinite locally finite graph, the $n$-step transition probabilities
  satisfy $p_n(x,y) \le C'/\sqrt{n}$ for some $C'>0$ and all $n\ge 1$,
  and, in particular, the probability to return to the origin
  satisfies $p_n(x,x) \le C'/\sqrt{n}$ for all $x$ and
  all $n\ge 1$; see Woess~\cite{woess:rw}*{Corollary 14.6}.

  This implies $\mu_1^{*n}(\stab_{G_1}(\rho_1))\le C_1/\sqrt n$ and
  $\mu_2^{*n}(\stab_{G_2}(\rho_2)) \le C_2/\sqrt n$ for some constants
  $C_1,C_2$ depending on $X_1,X_2$ and all $n\ge 1$.  Therefore,
  $\mu^{*n}(\stab_G(\rho)) \le C/n$, for $C=C_1C_2$ and all $n\ge1$.

  Consider $S=S_1 \cup S_2$. Clearly, $S$ is a generating set of
  $G=G_1 \times G_2$ whenever $S_1$ and $S_2$ are generating sets of
  $G_1$ and $G_2$ respectively, and
  $\|(g_1,g_2)\|_S=\|g_1\|_{S_1}+\|g_2\|_{S_2}$ for all $g_1 \in G_1$,
  $g_2 \in G_2$.  For all $n\ge 0$, we have
  $L_{\mu,G,S}(n)=L_{\mu_1,G_1,S_1}(n)+L_{\mu_2, G_2,S_2}(n)$; so, by
  the assumptions of the corollary, $L_\mu(n)=L_{\mu,G,S}(n) \le C
  n^\alpha$.

  We may therefore apply Theorem~\ref{thm:expwordgrowth} with $\delta=1$.
\end{proof}

Groups satisfying the assumption of the theorem admit a symmetric
measure of finite first moment whose boundary is non-trivial. However,
there are groups of exponential growth, such as for example wreath
products of a finite group with $\Z$, on which any symmetric finite
first moment measure has trivial boundary.

\begin{remark}
  The assumption that $X$ is a direct product is important, and is
  used to bound from above the return probabilities to the
  origin. There are examples of wreath products with infinite $X$,
  such as the group $A\wr_{X_1}\Grig$ studied
  in~\cite{bartholdi-erschler:permutational}, that have intermediate
  word growth and therefore trivial boundary for all measures of
  finite first moment.
\end{remark}

\begin{example}\label{ex:d=3}
  Consider $G=G_1\times G_2\times G_3$ and $X=X_1\times X_2\times X_3$
  with all $X_i$ infinite, transitive $G_i$-spaces. Then all
  permutational wreath products $A\wr_X G$ have exponential word
  growth, without any assumption on the $\mu_i$. Indeed, all simple
  random walks on these groups have a non-trivial boundary, as follows
  from Proposition~\ref{prop:transientimpliesnottriv}.
\end{example}

\begin{remark}\label{growthanddrift}
  Let $G$ be a group with word growth $v(n)$ at most $\exp(n^\beta)$ for
  some $\beta<1$, and let $\mu$ be a finitely supported measure on
  $G$. Then $L_{G,\mu}(n) \le C n^{(1+\beta)/2}$.
\end{remark}
\begin{proof}
  For any symmetric finitely supported random walk on a group $G$,
  there exists $K>0$ such that $L(n) \le K \sqrt{n \log v(n)+\log
    (n)}$ for all $n$, see~\cite{erschler:drift}*{Lemma~7.(ii)}.
\end{proof}

\begin{example}\label{ex:firstgrigorchuk}
  Consider $G_1$ and $G_2$ both equal to the first Grigorchuk group
  $\Grig$, and $X_1$ and $X_2$ some orbits for the action on the
  boundary of the rooted tree.  Recall that $\Grig$ has subexponential
  word growth, and more precisely by~\cite{grigorchuk:growth} has
  growth at most $\exp(n^\beta)$ for some $\beta<1$. The best known
  upper bound is $\beta=\log(2)/\log(2/\eta)\cong0.7674$ with
  $\eta^3+\eta^2+\eta=2$, see~\cite{bartholdi:upperbd}.  In view of
  Remark~\ref{growthanddrift}, the assumptions of
  Corollary~\ref{cor:mainB} are satisfied for $\alpha=(1+\beta)/2$, so
  $A\wr_{X_1\times X_2}(G_1\times G_2)$ has exponential growth as soon
  as $A$ is not trivial.
\end{example}

Among the Grigorchuk groups, there are groups with growth arbitrarily
close to exponential along a subsequence~\cite{grigorchuk:gdegree},
and in particular not bounded from above by any function of the form
$\exp(n^\alpha)$. We cannot use Remark~\ref{growthanddrift} to
estimate the drift of simple random walks on such groups. However,
every Grigorchuk group admits a finitely supported random walk whose
drift function is bounded from above by $Cn^\alpha$ for $\alpha=3/4$,
see Corollary~\ref{driftgrigorchuk}.

\section{A sufficient condition for triviality of the Poisson-Furstenberg boundary}\label{sec:trivboundary}
It is well known that the triviality of the boundary of an ordinary
wreath product of $A\wr G$ is related to the recurrence of $G$, see
Kaimanovich and
Vershik~\cite{kaimanovich-v:entropy}*{Proposition~6.4}. However, their
argument does not seem to provide information about the triviality of
the boundary in the case of a permutational wreath product $A\wr_XG$, in which
the action of $G$ on $X$ is recurrent. Indeed, let $W=A\wr_XG$ be a
permutational wreath product, let $\rho$ be a point of $X$ and let
$W'$ be the subgroup of $W$ that projects to the stabilizer of $\rho$
in $G$.  Starting with a random walk on $W$ which induces a recurrent
random walk on $X$, we can claim (by renormalizing the random walk at
stop times in the stabilizer of $\rho$) that the boundary of this
random walk is equivalent to the boundary of some (in general,
infinitely supported) random walk on $W'$; however, in contrast with
the ordinary wreath product case the group $W'$ may be large, even if
$A$ is small.

Another approach to criteria for triviality of the boundary in the case of
ordinary wreath products $A\wr G$, in which the induced random walk on
$G$ is recurrent, is to estimate the entropy of the random
walk~\cite{dyubina:characteristics}.  The proposition below is an
analogue of such a criterion, but now in the case of permutational
wreath products. The main difficulty of the proof of this proposition,
which does not appear in the case of ordinary wreath products, is in
the estimation of the number of choices of the inverted orbit, see
Remark~\ref{rem:inv-direct}.

\begin{proposition}\label{prop:recurrenceimpliestriv}
  Let $A,G$ be groups of subexponential word growth, and set
  $W:=A\wr_X G$. Let $\mu$ be finitely supported probability measure
  on $W$.

  If the expected inverted orbit growth of the projected random walk
  $(X,\overline\mu)$ grows sublinearly, then $h(\mu)=0$; so the random
  walk on $(W, \mu)$ has trivial Poisson-Furstenberg boundary.
\end{proposition}

\begin{lemma}\label{lem:subexp+}
  Let $G$ be a group of subexponential word growth, and let
  $\delta\colon \N\to\N$ be a sublinear function. Then the function
  \[v^+(n):=\#\{(g_1,\dots,g_k)\in G^k\mid k\le \delta(n), \|g_1\|+\cdots+\|g_k\|\le n\}\]
  grows subexponentially.
\end{lemma}
\begin{proof}
  Let $v(n)$ denote the growth function of $G$; then by hypothesis,
  for every $\epsilon>0$, there exists $C$ such that $\log
  v(n)\le\epsilon n+C$. We then estimate
  \begin{align*}
    v^+(n) &=\sum_{0\le m\le n}\sum_{0\le k\le \delta(n)}\sum_{n_1+\dots+n_k=m}v(n_1)\cdots v(n_k)\\
    &\le n\delta(n)\binom{n+\delta(n)-1}{\delta(n)}\max_{n_1+\dots+n_{\delta(n)}=n}v(n_1)\cdots v(n_{\delta(n)}).
  \end{align*}
  Let us show that the binomial \coefficient\
  $\binom{n+\delta(n)-1}{\delta(n)}$ is subexponential when $\delta$
  is sublinear. We use the following simple approximation for binomial
  \coefficient s, which comes from Stirling's formula for $n!$:
  \[\binom nk\approx\sqrt{\frac{2\pi n}{k(n-k)}} \left(\frac kn\right)^{-k}\left(\frac{n-k}n\right)^{k-n},\]
  in the sense that the quotient tends to $1$ as $n,k\to\infty$. In
  particular, if $k\le n/2$ then $\frac 1n\log\binom nk\le
  2\frac{-k}n\log(k/n)+\frac{-\log(2\pi)}{n}\log(k/n)\to0$ as
  $k/n\to0$. Therefore,
  \[\lim_{n\to\infty}\frac1n\log v^+(n)\le\lim_{n\to\infty}\frac1n\log\binom{n+\delta(n)-1}{\delta(n)}+
  \lim_{n\to\infty}\frac1n\big(\epsilon n+C\delta(n)\big)=\epsilon.\]
  Since this holds for all $\epsilon>0$, we have $\lim_{n\to\infty}\frac1n\log v^+(n)=0$.
\end{proof}

\begin{proof}[Proof of Proposition~\ref{prop:recurrenceimpliestriv}]
  We will show, for every $n$, that with positive probability a
  length-$n$ random walk lands in a subset of $W$ of subexponential
  size in $n$.

  Since $\mu$ is finitely supported, there exists a finite set
  $Y\subseteq X$ and a finite set $S\subset A$, which we may assume is
  generating, such that $\supp(\mu)\subseteq\sum_YS\times G$; namely,
  the random walk modifies only positions in $Y$, and does at most a
  step in $S$ at these positions.  Let $\delta(n)$ be the expectation
  of the inverted orbit growth of $(G,X)$, starting at all positions
  in $Y$. By assumption, $\delta$ grows sublinearly.

  We restrict ourselves to length-$n$ trajectories $\Omega_n$ whose
  inverted orbit visits less than $2\delta(n)$ points. These describe
  a subset of trajectories of measure at least $1/2$: indeed,
  $\expect[\delta(w)]=\delta(n)\le(1-\mu(\Omega_n))2\delta(n)$ whence
  $\mu(\Omega_n)\ge\frac12$.

  Let $\boldsymbol w=w_1\dots w_n\in\Omega_n$ be a
  trajectory. Considering simultaneously all $y\in Y$, the inverted
  orbit of $\boldsymbol w$ visits (a subset of) $\mathcal O=\{y
  w_{i(1)}\cdots w_n,\dots,y w_{i(k)}\cdots w_n\colon y\in
  Y\}\subseteq X$, say for definiteness at lexicographically minimal
  times $i(1),\dots,i(k)$; and $k\le 2\delta(n)$. This inverted orbit
  is determined by the sequence of group elements
  \[(w_{i(1)}\cdots w_{i(2)-1},w_{i(2)}\cdots
  w_{i(3)-1},\dots,w_{i(k)}\cdots w_n)\in G^k.
  \]
  By Lemma~\ref{lem:subexp+}, there is a subexponential number
  $v^+(n)$ of possibilities for $\mathcal O$ that may occur.

  Once a subset $\mathcal O$ of $X$ is chosen, let us consider the
  endpoint $\boldsymbol w=(f,g)$ of the trajectory, with $g\in G$ and
  $f\in\sum_XA$. The support of $f$ is contained in $\mathcal O$, and
  the random walk did a total of at most $|Y|n$ steps at positions in
  $\mathcal O$. Let $u(n)$ denote the growth function of $A$, by
  assumption subexponential. Assume the random walk did $n_x$ steps at
  each $x\in\mathcal O$, with $\sum_{x\in\mathcal O}n_x\le|Y|n$. Then
  $f\in\prod_{x\in\mathcal O}B_A(n_x)$, which is a subset of $\sum_XA$
  of subexponential growth, again by Lemma~\ref{lem:subexp+}.

  Since a product of subexponential functions is again subexponential,
  $\boldsymbol w$ belongs to a set of subexponential growth, when
  $\mathcal O$ ranges over all possible inverse orbits of trajectories
  in $\Omega_n$.

  Finally, to estimate the asymptotic entropy of $\mu$, it suffices to
  compute it on a subset of trajectories of positive measure. Indeed,
  consider $\epsilon>0$ and subsets $\Theta_n\subset W^n$ with
  $\mu^n(\Theta_n)\ge\epsilon$. If $h(\mu)=h>0$, then
  $\lim\frac{-1}n\log\mu^{*n}(\boldsymbol w)=h$ for almost every
  trajectory $\boldsymbol w\in G^\infty$,
  by~\cite{kaimanovich-v:entropy}*{Theorem~2.1}; so
  \[-\sum_{\substack{g\in W\\g=g_1\dots g_n\\(g_i)\in\Theta_n}}\mu^{*n}(g)\log\mu^{*n}(g)=h\epsilon>0.
  \]
  It therefore suffices, as we have done, to show that the asymptotic
  entropy of $\mu$ vanishes on a subset of positive measure.
\end{proof}

The proposition implies that the symmetric finitely supported random
walk on $W$ from Example~\ref{ex:firstgrigorchuk} has trivial
boundary.  Indeed, any nearest neighbour random walk on $\Z_+^2$ or
$\Z^2$ is recurrent~\cite{baldi-l-p:recurrents}; this property depends
only on the graph, not the random walk, see the remark before
Corollary~\ref{cor:mainA}. Note also that a subgraph of a recurrent
graph is also recurrent (see again~\cite{baldi-l-p:recurrents},
or~\cite{woess:rw}*{Corollary~2.15}), so we need not worry whether the
random walk is degenerate or intransitive.

\begin{remark}\label{rem:inv-direct}
  Implicit in the application of Lemma~\ref{lem:subexp+} is the
  following function $v_i(n,k)$ that deserves further study: for a group
  $G$, with generating set $S$, acting on a set $X$ with basepoint
  $\rho$, write
  \begin{multline*}
    v_i(n,k) = \#\{Y\subset X\mid\#Y=k\\\text{ and $Y$ is the inverted
      orbit of an $S$-path of length $n$ starting at }\rho\}.
  \end{multline*}
  Indeed, the lemma was used to show that, if $G$ is a group of
  subexponential growth with sublinear inverted orbit growth
  $\delta(n)$, then $v_i(n,\delta(n))$ is subexponential.

  By comparison, consider the corresponding function for direct orbits:
  \begin{multline*}
    v_d(n,k) = \#\{Y\subset X\mid\#Y=k\\\text{ and $Y$ is the direct
    orbit of an $S$-path of length $n$ starting at }\rho\}.
  \end{multline*}
  Each directed orbit $Y$ is a connected subset of $X$ containing
  $\rho$; and a connected subset of cardinality $k$ can be traversed
  by a path of length $2k$, so we have the simple bound
  $v_d(n,k)\le(\#S)^{2k}$, which implies that $v_d(n,k)$ is
  subexponential in $n$ as soon as $k$ is sublinear in $n$.

  In contrast with the direct orbit case, it is not possible in
  general to bound $v_i(n,k)$ by a function of $k$ only. For example,
  consider the first Grigorchuk group $G$ acting on a ray $X$. The
  stabilizer of $\rho$ is infinite; let $S$ contain the generating set
  of an infinite subgroup of it. If $s_2,\dots,s_n$ fix $\rho$ but
  $s_1$ does not, then the inverted orbit of $s_1\dots s_n$ contains
  only two points $\{\rho,\rho'\}$, and $\rho'$ is arbitrary under the
  condition $d(\rho,\rho')\le n$; so $v_i(n,2)\sim n$.

  More generally, we have the obvious bound $v_i(n,k)\le
  \#B_X(n)^k$. This bound is never tight enough for our purposes.
%
%
\end{remark}

We now show that in this example (or, more generally, in any torsion
Grigorchuk group) the assumption that the random walk is symmetric can
be dropped, see Corollary~\ref{cor:mainA}.

\subsection{Centered Markov chains}
There is a class of non-symmetric (and not necessarily reversible)
Markov chains that resembles in many aspects symmetric ones. These are
chains that admit a certain ``decomposition into cycles'',
see~\cite{kalpazidou:book}.  In particular, it is shown by Kalpazidou
in~\cite{kalpazidou:onmultiple} that under some conditions the
recurrence of such random walks does not depend on the choice of the
random walk. We will use a version of this statement which is due to
Mathieu.

\begin{definition}[\cite{mathieu:carnevaropoulos}*{Definition 2.1}]
  Let $V$ be an oriented graph, possibly with loops and multiple
  edges. A {\it centered Markov chain} on $V$ is defined as
  follows. There is a collection of $\{\gamma_i\}$ of oriented cycles
  on $V$, which we assume edge-self-avoiding but not necessarily
  vertex-self-avoiding. Each cycle has a weight $q_i$. Each edge must
  belong to exactly one cycle (but remember, we allow multiple
  edges!).  For any vertex $x$ in $V$, the sum of the weights of all
  cycles passing through $x$ (counted with multiplicity, if the cycle
  passes several times through $x$) is equal to one.
\end{definition}

The Markov chain has the vertex set of $V$ as set of states. The
probability of moving in one step of the Markov chain from vertex $x$
to vertex $y$ is given as follows: choose a cycle containing $x$
according to the weights $q_i$; then move to the successor of $x$
along that cycle. We write the transition kernel as follows:
\[q(x,y)=\sum_{i:(x,y)\in\gamma_i}q_i.\]

(The definition above is a particular case of
\cite{mathieu:carnevaropoulos}*{Definition~2.1}, and is slightly more
general than \cite{mathieu:carnevaropoulos}*{Example~2.4}. Indeed,
observe that under our assumption $q_i\le 1$ and, in the notation
of~\cite{mathieu:carnevaropoulos}, we can consider $m(x)=1$ for all
$x\in V$).

Centered Markov chains are generalizations of symmetric Markov chains:
indeed, in any non-oriented graph, replace each edge by two oriented
edges that form a cycle of length two; set the weight of that cycle to
be the weight of the original edge. In fact, the general definition
of centered Markov chains
in~\cite{mathieu:carnevaropoulos}*{Definition~2.1} is a
generalization of reversible Markov chains.

\begin{remark} 
  (i) If $\mu$ is a finitely supported measure on a group $G$ and all
  elements of the support of $\mu$ are torsion, then the random walk
  $(G,\mu)$ is a centered Markov chain on the Cayley graph of $G$ with
  generating set $\supp(\mu)$. This is used
  in~\cite{mathieu:carnevaropoulos} to prove Carne-Varopoulos
  estimates for random walks on torsion groups.

  (ii) More generally, if $\mu$ is a finitely supported measure on a
  group $G$ and all elements of the support of $\mu$ are torsion, and
  $G$ acts on a set $X$, then the random walk on $X$ is a centered
  random walk on the Schreier graph of $(G,X)$.
\end{remark}
\begin{proof}
  For each $g \in \supp(\mu)$ there exists a minimal $m\ge 1$ such that
  $g^m=1$. For each such $g$, consider all the cycles of the form $(x,
  x g, \dots x g^{m-1})$, and define the weight of this cycle to be
  $\mu(g)$. The random walk on $X$ induced by the measure $\mu$ is the
  same as the centered Markov chain defined by these weighted cycles.
\end{proof}

\begin{lemma}[Mathieu, \cite{mathieu:carnevaropoulos}*{Proposition 2.13(iii)}]\label{lem:mathieu}
  Let $V$ be a connected locally finite graph, and let $q$ be a
  centered Markov chain on $V$. Let $q_0$ be the associated symmetric
  Markov chain: $q_0(x,y)=\frac12(q(x,y)+q(y,x))$.

  Then the chain $q$ is recurrent if and only if $q_0$ is
  recurrent. 
\end{lemma}

Recall that a random walk is \emph{uniformly irreducible} if its
one-step transition probabilities are uniformly bounded from below.
It is well known (see e.g.~\cite{woess:rw}*{Corollary~3.5}) that if
$V^0$ is a non-oriented graph, the recurrence/transience of uniformly
irreducible symmetric random walks on $V^0$ does not depend on the
probability measure, but only on the graph. Thus, by the lemma, if $V$
is recurrent considered as a non-oriented graph, then all centered
Markov chains on $X$ are recurrent.

The following example (statements~(i) and~(ii) of the corollary below)
gives a negative answer to the question of Kaimanovich and Vershik
from~\cite{kaimanovich-v:entropy}:

\begin{corollary}[= Theorem~\ref{thm:mainA}]\label{cor:mainA}
  Let $G$ be any Grigorchuk torsion group, for example $\Grig$, and
  let $X$ be an orbit of $G$ on the tree's boundary. Let $A$ be a
  non-trivial finite group, and set $W=A\wr_{X\times X}(G\times
  G)$. Then
  \begin{enumerate}
  \item[(i)] $W$ has exponential word growth;
  \item[(ii)] any finitely supported random walk on $W$ has trivial
    Poisson-Furstenberg boundary;
  \item[(iii)] any finitely supported random walk on $W$ has zero drift.
  \end{enumerate}
\end{corollary}

\begin{proof} 
  (i) follows from Theorem~\ref{thm:expwordgrowth}; see its
  application in Example~\ref{ex:firstgrigorchuk}.

  For (ii), take a finitely supported measure $\mu$ on $G\times G$ and
  consider the induced random walk on $X^2$. Since $W$ is a torsion
  group, the random walk on $X^2$ is a centered Markov
  chain. By~\cite{bartholdi-g:spectrum}*{\S5.1}, the graph $X$ is
  $\Z_+$ or $\Z$, so the random walk on $X^2$ is a centered Markov
  chain on a graph which is a subgraph of $\Z^2$. By
  Lemma~\ref{lem:mathieu}, this random walk is recurrent.  Therefore,
  we can apply Proposition~\ref{prop:recurrenceimpliestriv}, and
  conclude that the random walk $(W,\mu)$ has trivial boundary. See
  Figure~\ref{fig:schreier} for a portion of the Schreier graph
  $X\times X$.

  For (iii), take a finitely supported measure $\mu$ on $W$. Any
  finitely supported measure has finite entropy. We have shown in~(ii)
  that the random walk $(W,\mu)$ has trivial boundary. Therefore, by
  the entropy criterion, $h(\mu)=0$.

  Mathieu has proven in~\cite{mathieu:carnevaropoulos} that
  Carne-Varopoulos estimates hold for centered Markov chains. In
  particular, he has shown that, for centered random walks on groups,
  $h(\mu)=0$ if and only if $\ell(\mu)=0$. We conclude that
  $\ell(\mu)=0$.
\end{proof}

\section{Further examples. Drift estimates for self-similar random walks}\label{sec:grig}
Self-similar groups are groups $G$ endowed with a homomorphism
$\phi\colon G\to G\wr\sym_d$, with $\sym_d$ the symmetric group on
$\{1,\dots,d\}$. By iterating the map $\phi$, every self-similar group
acts on sets of cardinality $d^n$, for all $n\in\N$; these sets form
the levels of a $d$-regular rooted tree. If we write
$\phi(g)=\pair{g_1,\dots,g_d}\pi$, then the permutation $\pi\in\sym_d$
describes the action of $g$ on the neighbours of the root, while
$g_1,\dots,g_d$ describe recursively the action of $g$ on the subtrees
attached to the root.

A fundamental example is the first Grigorchuk group $\Grig$. It is the
self-similar group characterized as follows: it is generated by four
elements $a,b,c,d$; it acts faithfully on the $2$-regular rooted tree;
and $\phi$ is given on the generators by
\begin{equation}\label{eq:grigorchuk}
  \phi(a)=\pair{1,1}(1,2),\quad\phi(b)=\pair{a,c},\quad\phi(c)=\pair{a,d},\quad\phi(d)=\pair{1,b}.
\end{equation}

Self-similar random walks were introduced by the first author and
Virag in~\cite{bartholdi-v:amenability}; see below for the
definition. In that paper, they show that the so-called ``Basilica
group'' admits a self-similar random walk, and then this self-similar
measure is used to show that this random walk has zero drift with
respect to some metric (which is not a word metric, in contrast with
usual definition of the drift).

Kaimanovich uses a similar idea in~\cite{kaimanovich:munchhausen}, but
works with the entropy of the random walk $h(\mu)$ instead.  The main
idea of these papers is to use the self-similarity of the random walk
to prove that its asymptotic entropy vanishes. In a similar way one
can use self-similar measures in order to estimate $H_\mu(n)$, see
\cite{bartholdi-k-n-v:ba}*{Proposition 4.11}.  The following lemma
is similar to that proposition.

\begin{definition}
  A \emph{self-similar sequence of groups} is a sequence
  $(G_1,G_2,\dots)$ of groups, with homomorphisms $\phi_i\colon G_i\to
  G_{i+1}\wr\sym_d$.

  Let $\mu_i$ be a measure on $G_i$. It defines a random walk on
  $G_{i+1}\times\{1,\dots,d\}$, via $\phi_i$: if
  $\phi_i(g)=\pair{g_1,\dots,g_d}\pi$, then the walk moves from
  $(h,i)$ to $(h g_i,\pi(i))$ with probability $\mu_i(g)$. The
  \emph{renormalization} of $\mu_i$ is the measure $\mu'_i$ on
  $G_{i+1}$ defined by running $\mu_i$ on $(1,1)$ till it reaches
  $G_{i+1}\times1$; in formulas,
  \[\mu'_i(g)={\kern6mm\sum_{\kern-6mm h_1,\dots,h_n\in G_{i+1}\kern-6mm}}'\kern6mm\mu_i(h_1)\cdots\mu_i(h_n),\]
  where the sum extends over all $n$-tuples $(h_1,\dots,h_n)$ such
  that $\phi_i(h_1\cdots h_n)\in\{g\}\times
  G_{i+1}^{d-1}\times\stab_{\sym_d}(1)$ and $\phi_i(h_1\cdots h_j)\notin
  G_{i+1}^d\times\stab_{\sym_d}(1)$ for all $j<n$.

  A \emph{self-similar sequence of measures} on a sequence $(G_i)$ of
  groups is a sequence $(\mu_1,\mu_2,\dots)$ of measures, each $\mu_i$
  a measure on $G_i$, and a sequence of numbers
  $(\alpha_1,\alpha_2,\dots)$ in $[0,1]$, such that
  $\mu'_i=(1-\alpha_i)\delta_1+\alpha_i\mu_{i+1}$, namely $\mu'_i$ is
  a convex combination of $\mu_{i+1}$ and the Dirac measure at $1\in
  G_{i+1}$. It is a \emph{lazy} random walk, with \emph{laziness}
  $\alpha_i$.
\end{definition}

\noindent The following lemma
generalizes~\cite{bartholdi-k-n-v:ba}*{Proposition~4.11}; see also~\cites{brieussel:entropy,amir-virag:speedexponents}:
\begin{lemma}\label{lem:selfsimilar}
  Let $(G_i)$ be a self-similar sequence of groups, and let $(\mu_i)$
  be a self-similar sequence of measures on $(G_i)$, with laziness
  $(\alpha_i)$. Assume $\sup_i H(\mu_i)<\infty$. Then there exists a
  constant $K$ such that
  \[H_{G_1,\mu_1}(n)\le K n^{\beta}\text{ for all }n,\text{ with
  }\beta=\frac{\log d}{\log d-\log(\sup\alpha_i)}.\]
\end{lemma}

\begin{proof} A random variable in $G_i$ is determined by its
  projection to $\sym_d$ and by its $d$ renormalizations in
  $G_{i+1}$. Say an $n$-step walk starting at $1$ visits $n_i$ times
  point $i$, for all $i\in\{1,\dots,d\}$. Then
  \[H_{\mu_i}(n) \le \expect[\sum_{j=1}^d  H_{\mu_i'}(n_i)\mid n_1+n_2+\dots +n_d \le n]+d\log d.
  \]
  For $\nu$ a measure, we extend $H_\nu(n)$ to real arguments $n\in\R$
  by interpolating linearly:
  $H_\nu((1-\theta)n+\theta(n+1))=(1-\theta)H_\mu(n)+\theta
  H_\mu(n+1)$. By~\cite{kaimanovich-v:entropy}*{Proposition~1.3}, for
  $n\in\N$ the numbers $H_\nu(n+1)-H_\nu(n)$ decrease monotonically to
  $h(\nu)$; so the affine extension $H_\nu\colon \R\to\R$ is a concave
  function. Therefore,
  \[H_{\mu_i}(n) \le d H_{\mu'_i}(n/d)+d\log d.\]
  Next, for all $m\in\N$,
  \[H_{\mu'_i}(m)=\sum_{k=0}^m\binom mk\alpha_i^k(1-\alpha_i)^{m-k}H_{\mu_{i+1}}(k);\]
  the binomial distribution has mean $\alpha_i m$, so again by concavity
  \[H_{\mu'_i}(m)\le H_{\mu_{i+1}}(\alpha_i m).\]
  Therefore,
  \[H_{\mu_i}(n)\le d H_{\mu_{i+1}}(n\alpha_i/d)+d\log d.\]
  We then iterate this relation, to obtain
  \[H_{\mu_1}(d^k/\alpha_1\cdots\alpha_k)\le d\log d+\cdots+d^k\log d+d^k H(\mu_{k+1})\le K d^k\]
  for a constant $K$, and we are done.
\end{proof}

For any finitely supported random walk on a finitely generated group
$G$, there exist constants $C,D>0$ such that
\begin{equation}\label{eq:drift}
  C\left(\frac{L(n)}{n}\right)^2 \le \frac{H(n)}{n} \le D\frac{L(n)}{n}
\end{equation}
for all $n$; the first inequality follows from Varopoulos's long range
estimates, see e.g.~\cite{erschler:drift}*{page~1201}. These
inequalities hold, more generally, for any random walk with finite
second moment, see~\cite{erschler-karlsson:homo}*{Corollary~9.(ii)}.

Kaimanovich observes in~\cite{kaimanovich:munchhausen} that the first
Grigorchuk group admits a self-similar measure $\mu$ with laziness
$1/2$.  An example of such a measure is $\mu$ defined by
$\mu(1)=5/12$, $\mu(a)=1/3$, $\mu(b)=\mu(c)=\mu(d)=1/12$: for this
measure one has $\mu'=1/2 \delta_1 + 1/2 \mu$.  We can therefore apply
Lemma~\ref{lem:selfsimilar} with $d=2$ and $\alpha=1/2$ and conclude
that the entropy function $H_\mu(n)$ of this random walk satisfies
$H(n) \le K n^{1/2}$.

Now take a sequence $\omega =
(v_1,v_2,\dots)\in\{\mathbf0,\mathbf1,\mathbf2\}^\infty$, and define
$\omega_i = (v_i, v_{i+1}, \dots)$ its shift; consider the
corresponding sequence of Grigorchuk groups $G_i=G_{\omega_i}$, which
form a similar sequence.  The standard generators of $G_i$ are still
written $a,b,c,d$. On each $G_i$ define a probability measure $\mu_i$
by $\mu_i(1)=5/12$, $\mu_i(a)=1/3$,
$\mu_i(b)=\mu_i(c)=\mu_i(d)=1/12$. These form a self-similar sequence
of measures on $(G_i)$, and, as in~\cite{kaimanovich:munchhausen}, one
has $\mu_i' = 1/2 \delta_1 +1/2 \mu_{i+1}$. Combining this with
Lemma~\ref{lem:selfsimilar} and~\eqref{eq:drift}, we get the
\begin{corollary}\label{driftgrigorchuk}
  On every Grigorchuk group $G_\omega$, there exists a symmetric
  non-degenerate finitely supported measure $\mu$ and a constant $C$
  such that $H(n)\le C n^{1/2}$ and $L(n) \le C n^{3/4}$ for all
  $n\in\N$.
\end{corollary}

\begin{remark}
  Examples of Grigorchuk groups above stress the importance of the
  fact that~\cite{bartholdi-v:amenability} works with drift with
  respect to a special non-word metric,
  and~\cite{kaimanovich:munchhausen} works with entropy of random
  walks, and not with drift: although Grigorchuk groups admit
  self-similar measure sequences with laziness $1/2$, it is not true
  that on these groups one has $L(n) \le C n^{1/2}$. Indeed, it is
  shown in~\cite{erschler:critical}*{Corollary~1} that any simple
  random walk on the first Grigorchuk group satisfies $L(n) \ge
  n^\kappa$ for some $\kappa>1/2$ and infinitely many $n$'s.
\end{remark}

\begin{example}\label{ex:grigorchukgeneral}
  Let $G_1,G_2$ be two Grigorchuk groups.  Let respectively $X_1,X_2$
  be orbits for their action on the boundary of the rooted tree.  By
  Corollary~\ref{driftgrigorchuk}, the assumption of
  Theorem~\ref{thm:expwordgrowth} is satisfied.  Therefore, for any non-trivial
  group $A$, the wreath product $W= A \wr_{X_1\times X_2} G_1\times
  G_2$ has exponential word growth.

  If $G_1$ and $G_2$ are torsion groups, then every finitely supported
  measure on $W$ has trivial boundary, so these are other negative
  answers to the Kaimanovich-Vershik question.
\end{example}

\begin{example}\label{ex:grigorchuktf}
  Let $G_1=G_2=H$ be the Grigorchuk torsion-free group of
  subexponential growth from~\cite{grigorchuk:pgps}; recall that
  $H$ maps onto $\Grig$, and therefore acts on an orbit $X$ of the
  Grigorchuk group on the boundary of the rooted tree. Consider the
  wreath product $W=\Z\wr_{X\times X}(H\times H)$. Then $W$ is a
  torsion-free group of exponential growth, such that every finitely
  supported measure on $W$ has trivial Poisson-Furstenberg boundary.
\end{example}
\begin{proof}
  Clearly $W$ is torsion-free, as an extension of torsion-free
  groups. Since the action of $H\times H$ on $X\times X$ actually
  comes from the action of $\Grig\times\Grig$, the random walk $\mu$
  on $X\times X$ induced by $H\times H$ is the same as a random walk
  induced by a measure on $\Grig\times\Grig$. Therefore, $\mu$ defines
  a centered random walk on a subgraph of $\Z^2$.  Applying
  Lemma~\ref{lem:mathieu} as we did in the proof of
  Corollary~\ref{cor:mainA}(ii), we conclude that $\mu$ induces a
  recurrent random walk. By Lemma~\ref{lem:expectedinvertedorbit}, the
  expected inverted orbit growth is sublinear. Since both $\Z$ and
  $H\times H$ have subexponential growth,
  Proposition~\ref{prop:recurrenceimpliestriv} gives that every
  finitely supported measure on $W$ has trivial boundary. On the other
  hand, $W$ has exponential word growth since, by
  Theorem~\ref{thm:expwordgrowth}, its quotient $\Z\wr_{X\times
    X}(\Grig\times\Grig)$ already has exponential growth.
\end{proof}

\section{Lipschitz imbeddings of regular trees}
We gave, in Theorem~\ref{thm:expwordgrowth}, a general criterion for a
permutational wreath product of a product of two groups to have
exponential word growth. For most of the examples we produce, it does
not seem at all straightforward to check without using random walks
that they have exponential growth.

Below is one example in which we prove more directly that the growth of an
inverted orbit of $(G,X)$ is linear (and hence that the word growth of the
corresponding wreath product $A\wr_X G$ is exponential). We consider
$G=\Grig$, acting diagonally on $X=X_1\times X_2$, where $X_1$ and $X_2$
are orbits under $\Grig$ of the rays $\rho_1 = (12)^\infty$ and $\rho_2=
(21)^\infty$ respectively (regarded as points of the boundary of the rooted
tree $\{1,2\}^*$ on which $\Grig$ acts).

\begin{proposition}\label{prop:directproof}  
  Let $w_n$ be the word over $\{a,b,c,d\}$ of length $\sim (2/\eta)^n$
  constructed as follows.  Write $\Omega'=\{ab,ac,ad\}^*\subset
  \Omega=\{a,b,c,d\}^*$, consider the substitution $\zeta\colon\Omega'\to
  \Omega'$ given by
  \[\zeta:ab\mapsto abadac,\quad ac\mapsto abab,\quad ad\mapsto acac,\]
  and consider the word $w_n=\zeta^n(ad)$.

  For a word $w=g_1\dots g_\ell$ define
  $\delta(w) = \#\{(\rho_1, \rho_2)g_i\cdots g_\ell\mid 0\le
  i\le\ell\}$ with $\rho_1 = (12)^\infty$ and $\rho_2 =
  (21)^\infty$. Then for all $n\ge1$ we have $\delta(w_n)=|w_n|+1$;
  namely, all points on the inverse orbit of $w_n$ are distinct.
\end{proposition}

The words $w_n$ in the statement of the lemma above were used
in~\cite{bartholdi-erschler:permutational}*{Proposition~4.7} to
estimate the growth of the permutational wreath product of the first
Grigorchuk group.

\begin{proof}
  Write $w_n =g_1\dots g_\ell$ and $\rho=(\rho_1,\rho_2)$. We are to
  show that for all $i<j$ we have $\rho g_i\cdots g_\ell \ne \rho
  g_j\cdots g_\ell$; or, equivalently, that $\rho g_i\cdots g_{j-1}
  \ne\rho$, namely, no subword of $w_n$ fixes $\rho$.

  Let $F$ denote the free product
  $\langle a\mid a^2\rangle*\langle b,c,d\mid b^2,c^2,d^2,bcd\rangle$;
  its elements may be identified with those words in $\Omega$ that
  alternate in `$a$' and `$b,c,d$' letters, and we have a natural
  quotient map $F\to G$. The map $\phi\colon G\to G\wr\sym_2$ lifts to
  a map $\phi\colon F\to F\wr\sym_2$ by the same defining
  formula~\eqref{eq:grigorchuk}. We note
  $\phi(w_n)=\pair{w_{n-1}^a,w_{n-1}}$ as an equality in $F$, and more
  generally if $u$ is a subword of $w_n$ then
  $\phi(u)=\epsilon^s\pair{u_1,u_2}$ with $-1\le|u_1|-|u_2|\le1$ and
  $u_1,u_2$ are subwords of $w_{n-1}$ which overlap on all except
  possibly one letter. Let $\eta\approx2.46$ be the positive root of
  $X^3-X^2-2X-4$; then $|u_1|\approx|u|/\eta$.

  Assume now for contradiction that a non-trivial subword
  $g_i\dots g_{j-1}$ of $w_n$ fixes $\rho$. Set
  $u_1^{(n)}=u_2^{(n)}=g_i\dots g_{j-1}$ and for all $m<n$ define
  subwords $u_1^{(m)},u_2^{(m)}$ of $w_m$ by
  $\phi(u_1^{(m+1)})=\pair{...,u_2^{(m)}}$ and
  $\phi(u_2^{(m+1)})=\pair{u_1^{(m)},...}$. Then
  $\rho_1 u_1^{(m)}=\rho_1$ and $\rho_2 u_2^{(m)}=\rho_2$ for all
  $m\in\{1,\dots,n\}$. Let $m$ be maximal such that $|u_1^{(m)}|\le4$
  and $|u_2^{(m)}|\le4$; in particular $u_1^{(m)}$ and $u_2^{(m)}$ are
  non-trivial. Now the only possibilities for a non-trivial word of
  length $\le4$ to fix $\rho_1$ or $\rho_2$ are
  $u_1^{(m)}\in\{d,aca,acad,daca\}$ and
  $u_2^{(m)}\in\{c,ada,adac,cada\}$, and none of these words have
  sufficient overlap.
\end{proof}

\def\mycolor#1{\if#1bred\else\if#1cgreen\else blue\fi\fi}

\begin{figure}
  \begin{center}
  \begin{tikzpicture}[xscale=0.44,yscale=0.44,decoration={snake,amplitude=.3mm,segment length=1mm}]
  \footnotesize
  \def\leftcoord{-10}
  \def\topcoord{9}
  \foreach\i in {-7,-5,...,8} \path (\i+0.2,\topcoord) edge[decorate] node[above] {$a$} +(0.8,0);
  \foreach\i/\g in {-8/b,-6/b,-4/b,-2/c,0/b,2/b,4/b,6/b,8/b}
    \path (\i,\topcoord) edge[bend left,\mycolor{\g}] node[above] {$\g$} +(1.2,0);
  \foreach\i/\g in {-8/c,-6/d,-4/c,-2/d,0/c,2/d,4/c,6/d,8/c}
    \path (\i,\topcoord) edge[bend right,\mycolor{\g}] node[below] {$\g$} +(1.2,0);
  \foreach\i/\g in {-8/d,-6.8/d,-6/c,-4.8/c,-4/d,-2.8/d,-2/b,-0.8/b,0/d,1.2/d,2/c,3.2/c,4/d,5.2/d,6/c,7.2/c,8/d,9.2/d}
    \path (\i,\topcoord) edge[loop above,\mycolor{\g}] node[above] {$\g$} ();
  \path (9.2,\topcoord) edge[densely dotted] +(0.5,0);
  \path (-8,\topcoord) edge[densely dotted] +(-0.5,0);
  \draw (0,\topcoord) node[below] {$\rho_1$};

  \foreach\j in {-7,-5,...,6} \path (\leftcoord,\j+0.2) edge[decorate] node[left] {$a$} +(0,0.8);
  \foreach\j/\g in {-8/b,-6/b,-4/b,-2/b,0/b,2/b,4/c,6/b}
    \path (\leftcoord,\j) edge[bend left,\mycolor{\g}] node[left] {$\g$} +(0,1.2);
  \foreach\j/\g in {-8/d,-6/c,-4/c,-2/c,0/d,2/c,4/d,6/c}
    \path (\leftcoord,\j) edge[bend right,\mycolor{\g}] node[right] {$\g$} +(0,1.2);
  \foreach\j/\g in {-8/c,-6.8/c,-6/d,-4.8/d,-4/d,-2.8/d,-2/d,-0.8/d,0/c,1.2/c,2/d,3.2/d,4/b,5.2/b,6/d,7.2/d}
    \path (\leftcoord,\j) edge[loop left,\mycolor{\g}] node[left] {$\g$} ();
  \path (\leftcoord,7.2) edge[densely dotted] +(0,0.5);
  \path (\leftcoord,-8) edge[densely dotted] +(0,-0.5);
  \draw (\leftcoord,0) node[right] {$\rho_2$};

  \foreach\i in {-6,-2,...,8} \foreach\j in {-6,-2,...,6}
    \path (\i,\j) edge[decorate] +(-0.8,-0.8);
  \foreach\i in {-4,0,...,8} \foreach\j in {-4,0,...,6}
    \path (\i,\j) edge[decorate] +(-0.8,-0.8);
  \foreach\i in {-4,0,...,8} \foreach\j in {-6,-2,...,6}
    \path (\i-0.8,\j) edge[decorate] +(0.8,-0.8);
  \foreach\i in {-6,-2,...,7} \foreach\j in {-4,0,...,5}
    \path (\i-0.8,\j) edge[decorate] +(0.8,-0.8);

  \foreach\i/\j/\a/\b/\c in {-8/-8/c/d/b,-8/0/c/d/b,-8/4/b/d/c,
    -6/-6/d/c/b,-6/-2/d/c/b,-6/2/d/c/b,-6/6/d/c/b,
    -4/-8/c/d/b,-4/0/c/d/b,-4/4/b/d/c,
    -2/-6/d/b/c,-2/-2/d/b/c,-2/2/d/b/c,-2/6/d/b/c,
    0/-8/c/d/b,0/0/c/d/b,0/4/b/d/c,
    2/-6/d/c/b,2/-2/d/c/b,2/2/d/c/b,2/6/d/c/b,
    4/-8/c/d/b,4/0/c/d/b,4/4/b/d/c,
    6/-6/d/c/b,6/-2/d/c/b,6/2/d/c/b,6/6/d/c/b,
    8/-8/c/d/b,8/0/c/d/b,8/4/b/d/c}
    \path (\i,\j) edge[\mycolor{\a}] (\i+1.2,\j)
          (\i,\j+1.2) edge[\mycolor{\a}] (\i+1.2,\j+1.2)
          (\i,\j) edge[\mycolor{\b}] (\i,\j+1.2)
          (\i+1.2,\j) edge[\mycolor{\b}] (\i+1.2,\j+1.2)
          (\i,\j) edge[\mycolor{\c}] (\i+1.2,\j+1.2)
          (\i+1.2,\j) edge[\mycolor{\c}] (\i,\j+1.2);
  \foreach\i/\j/\a/\b/\c in {-8/-4/d/b/c,
    -4/-4/d/b/c,
    0/-4/d/b/c,
    4/-4/d/b/c,
    8/-4/d/b/c}
    \path (\i,\j) edge[loop above,\mycolor{\a}] ()
          (\i,\j+1.2) edge[loop below,\mycolor{\a}] ()
          (\i+1.2,\j) edge[loop above,\mycolor{\a}] ()
          (\i+1.2,\j+1.2) edge[loop below,\mycolor{\a}] ()
          (\i,\j) edge[bend left=20,\mycolor{\b}] (\i+1.2,\j+1.2)
          (\i,\j) edge[bend right=20,\mycolor{\c}] (\i+1.2,\j+1.2)
          (\i+1.2,\j) edge[bend left=20,\mycolor{\b}] (\i,\j+1.2)
          (\i+1.2,\j) edge[bend right=20,\mycolor{\c}] (\i,\j+1.2);
  \draw(0,0) node[below] {$\rho$};
  \end{tikzpicture}
  \end{center}
  \caption{The Schreier graph of $\Grig$ on $\Grig/\stab(\rho_1)\times\Grig/\stab(\rho_2)$. Edges are indicated by colours: black for $a$, red/green/blue for $b/c/d$.}\label{fig:schreier}
\end{figure}
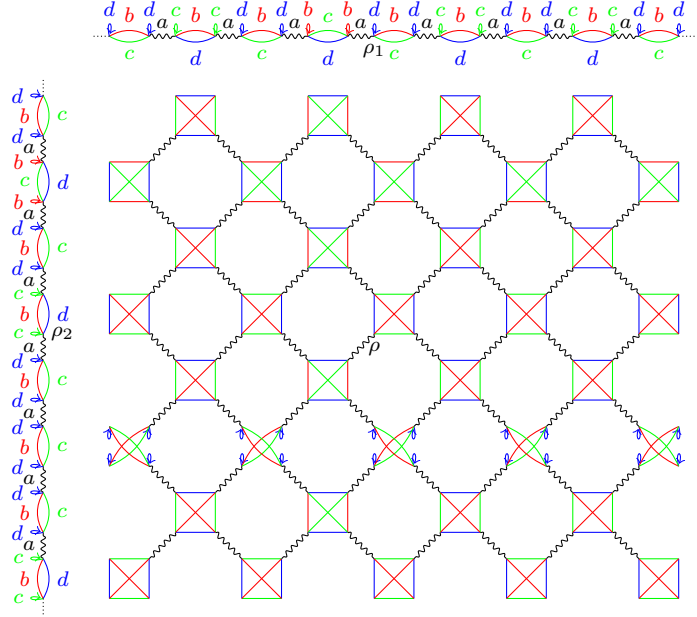

Let us tentatively introduce the following notion. Consider a group
$G$ acting transitively on a set $X$, and fix $\rho \in X$. Say that
the growth of inverted orbits of $G$ on $(X,\rho)$ is {\it strongly
  linear}, if there exists a finite generating set $S$ of $G$ such that
for each $n \in\N$ there exists a word $w_n$ of length $n$ over
elements of $S$ such that the inverted orbit of $w_n$ has exactly
$n+1$ points (recall that this is the maximal value it may assume).

Proposition~\ref{prop:directproof} shows that $(G,X)$ has strongly
linear growth.  Observe the following consequence of strongly linear
growth of inverted orbits:
\begin{lemma}\label{lem:trees}
  If $G$ has strongly linear inverted orbit growth on $X$ and $A$ is
  non-trivial, then some Cayley graph of $A\wr_X G$ contains an
  imbedded copy of the infinite binary rooted tree.
\end{lemma}
\begin{proof}
  Let $a_0\neq a_1$ be two elements of $A$. Let $S'$ be a generating
  set of $G$ for which the inverted orbits grow strongly linearly. Let
  $S$ be a generating set of $W:=A\wr_X G$ containing
  $\{a_0,a_1\}\times S'$. For $n\in\N$, let $w_n=g_1\cdots g_n$ be a
  word of length $n$ visiting $n$ points in $X$, and consider all
  words of the form $a_{i_m}g_m\cdots a_{i_n}g_n$ for all
  $m\in\{1,\cdots,n+1\}$ and all $i_m,\dots,i_n\in\{0,1\}$. We claim
  that these are the vertices of the height-$n$ binary rooted tree in
  the Cayley graph of $W$.

  First, these elements are all distinct: consider $a_{i_m}\cdots g_n$
  and $a_{i'_{m'}}\cdots g_n$. If $m\neq m'$ then their projections to
  $G$ are distinct; while if $m=m'$ then, because the inverted walk
  $g_m\cdots g_n$ visits $n-m+1$ distinct positions, the elements are
  distinct as soon as $i_j\neq i'_j$ for some $j\in\{m,\dots,n\}$.

  Because all $a_{i_m}g_m$ belong to $S$, there is an edge in the
  Cayley graph from $a_{i_m}g_m\cdots a_{i_n}g_n$ to
  $a_{i_{m+1}}g_{m+1}\cdots a_{i_n}g_n$; these edges form a binary
  tree, rooted at $1$.

  Since $n$ was arbitrary, we obtain for all $n$ a binary tree of
  height $n$ and rooted at $1$. A classical diagonal argument then
  extracts from this sequence an infinite binary rooted tree.
\end{proof}

\begin{corollary}\label{cor:wexp}
  The wreath product $W=A\wr_{X_1\times X_2}\Grig$ has exponential
  word growth, for $X_1$ the orbit of $\rho_1 = (12)^\infty$ and $X_2$
  is the orbit of $\rho_2= (21)^\infty$. Moreover, some Cayley graph
  of $W$ contains an infinite binary rooted tree.
\end{corollary}
It also follows from~\ref{thm:expwordgrowth} that $W$ has exponential
growth; indeed, $\Grig$ and $\Grig\times\Grig$ are commensurable, so
we are, up to finite index, in the situation of a product of groups
$G_1\times G_2$ acting on $X_1\times X_2$. We have elected to give a
direct proof that $W$ has exponential growth, because we also deduce
along the way that $W$ contains trees in its Cayley graph.

A classical question of Rosenblatt~\cite{rosenblatt:mixing} asks
whether every group of exponential growth admits a Lipschitz imbedding
of the infinite binary rooted tree.  A result of Benjamini and
Schramm~\cite{benjamini-schramm:trees} implies that every non-amenable
graph contains the image a regular tree by a Lipschitz imbedding; so
it is sufficient, to answer positively Rosenblatt's question, to
exhibit a non-amenable subgraph. Rosenblatt's question is answered
positively for virtually soluble groups (the group contains a free
subsemigroup) and non-amenable groups (since their Cayley graph is
non-amenable), but is open in general. The group $W$ we construct in
this article also contains Lipschitzly imbedded infinite binary rooted
trees (by Lemma~\ref{lem:trees}), though for a different reason than
those mentioned above.


\end{document}